\documentclass[pdflatex,sn-mathphys-num]{sn-jnl}
\RequirePackage{etex}
\usepackage[defaultsans]{droidsans}

\usepackage{main}
\graphicspath{{figures/}}
\allowdisplaybreaks

\begin{document}

\title{Approximate Bregman proximal gradient algorithm with variable metric Armijo line search}
\author[1]{\fnm{Kiwamu} \sur{Fujiki}}\email{fujiki-kiwamu1111@g.ecc.u-tokyo.ac.jp}
\author*[1]{\fnm{Shota} \sur{Takahashi}}\email{shota@mist.i.u-tokyo.ac.jp}
\author[1,2]{\fnm{Akiko} \sur{Takeda}}\email{takeda@mist.i.u-tokyo.ac.jp}

\affil*[1]{\orgdiv{Graduate School of Information Science and Technology}, \orgname{The University of Tokyo}, \orgaddress{\street{7-3-1 Hongo}, \city{Bunkyo}, \postcode{113-8656}, \state{Tokyo}, \country{Japan}}}
\affil[2]{\orgdiv{Center for Advanced Intelligence Project}, \orgname{RIKEN}, \orgaddress{\street{1-4-1 Nihonbashi}, \city{Chuo}, \postcode{103-0027}, \state{Tokyo}, \country{Japan}}}

\abstract{We propose a variant of the approximate Bregman proximal gradient (ABPG) algorithm for minimizing the sum of a smooth nonconvex function and a nonsmooth convex function. ABPG is known to converge globally to a stationary point even when the smooth part of the objective function does not have a globally Lipschitz continuous gradient, and its iterates can often be expressed in closed form. However, ABPG relies on an Armijo line search to guarantee global convergence, which can slow down its practical performance. To address this issue, we propose a variant of ABPG with a variable metric Armijo line search, which relaxes the Armijo condition. We prove that the line search terminates in finitely many steps for a general nonsmooth convex function. Under the variable metric Armijo condition, we establish global subsequential convergence of the algorithm. Moreover, assuming the Kurdyka--\L{}ojasiewicz property, we also prove that the algorithm globally converges to a stationary point. Numerical experiments on $\ell_p$-regularized least squares problems, nonnegative linear inverse problems and phase retrieval demonstrate that the proposed algorithm outperforms existing algorithms.}


\maketitle

\section{Introduction}
We consider composite nonconvex optimization problems of the form
\begin{align}
    \min_{x\in\cl C} \quad\Psi(x)\coloneqq f(x) + g(x), \label{problem}
\end{align}
where $f:\RR^n\rightarrow(-\infty,+\infty]$ is a continuously differentiable function, $g:\RR^n\rightarrow(-\infty,+\infty]$ is a possibly nondifferentiable convex function, and $\cl C$ is the closure of a nonempty open convex set $C \subset \RR^n$.
Optimization problems of the form~\cref{problem} arise in various applications, including the maximum a posteriori (MAP) estimate~\cite{MAP,MAP2}, ridge regression~\cite{ridge}, the least absolute shrinkage and selection operator (LASSO)~\cite{LASSO}. In machine learning and signal processing, regularization or penalty terms are often introduced to prevent overfitting and impose the model structure. Some regularization terms are not necessarily differentiable.

Numerous algorithms using proximal mappings have been proposed to solve~\cref{problem}. For instance, the proximal gradient method~\cite{Bruck1975-et,lions1979splitting,PASSTY1979383} and the fast iterative shrinkage-thresholding algorithm (FISTA)~\cite{fista} belong to the class of proximal algorithms.
Convergence analysis of these algorithms with constant step-sizes typically relies on the global Lipschitz continuity of $\nabla f$, i.e., there exists $L>0$ such that $\|\nabla f(x) - \nabla f(y)\|\leq L\|x-y\|$ for any $x,y\in\RR^n$. This condition is often restrictive and does not hold in certain applications in signal processing and machine learning.

Bolte \etal~\cite{Bolte-2018} proposed the Bregman proximal gradient algorithm (BPG). This algorithm globally converges under the smooth adaptable property~\cite{Bolte-2018}, also called relative smoothness~\cite{Lu-haihao}, which is a relaxation of the global Lipschitz continuity of $\nabla f$. In recent years, the Bregman proximal gradient method has been improved from various perspectives.
Hanzely~\etal~\cite{Hanzely2021-sz} proposed accelerated Bregman proximal gradient algorithms for convex optimization problems using the triangle scaling property.
Mukkamala~\etal~\cite{Mukkamala2020-ea} proposed an accelerated version of BPG.
Some researchers have applied Bregman proximal-type algorithms to linear inverse problems~\cite{Bauschke2017-hg,Takahashi}, nonnegative matrix factorization~\cite{Mukkamala2019-mk,Takahashi2026-rv}, and blind deconvolution~\cite{Takahashi2023-uh}.

Since the subproblem of BPG cannot always be solved in closed form and is sometimes hard to solve depending on the Bregman distance, Takahashi and Takeda~\cite{Takahashi} proposed the approximate Bregman proximal gradient algorithm (ABPG), whose subproblem is easier to solve. Instead of the Bregman distance, ABPG uses the approximate Bregman distance (see also \cref{App-Breg-dist}), which is the second-order approximation of the Bregman distance. The subproblem of ABPG can be written as the sum of a quadratic function and a regularizer. Moreover, if the Bregman distance and $g$ are separable, the subproblem of ABPG is reduced to $n$ independent one-dimensional optimization problems. ABPG uses the line search procedure to ensure the accuracy of the approximate Bregman distance.
However, the global convergence of ABPG has not been established when  $g\not\equiv 0$, and line search procedures can lead to slow convergence in practice.

In this paper, we propose a new algorithm, named the approximate Bregman proximal gradient algorithm with variable metric Armijo line search (ABPG-VMA). The line search procedure of this algorithm is inspired by variable metric inexact line search based algorithms~\cite{Bonettini2016-bn,Bonettini_2017}.
In the same way as ABPG, the subproblem of ABPG-VMA is defined by
\begin{align}
    y^k = \argmin_{u\in\cl C}\left\{\langle \nabla f(x^{k}), u- x^{k} \rangle + g(u) + \frac{1}{\lambda}\tilde{D}_{\phi}(u,x^{k})\right\},
\end{align}
where $x^k \in C$, $\lambda > 0$, and $\tilde{D}_\phi(u,x) := \frac{1}{2}\langle\nabla^2\phi(x)(u - x), u - x\rangle$ is the approximate Bregman distance with a twice continuously differentiable convex function $\phi$. The search direction of ABPG is defined by $d^k = y^k - x^k$, and ABPG searches for $t_k \in (0,1]$ in each iteration to decide the step-size. The Armijo-like condition adopted in~\cite{Takahashi} is so stringent that computing $t_k$ can be time-consuming, and it may force $t_k$ to be very small, which causes slow convergence. We adopt a relaxed condition and allow larger values of $t_k$, including $t_k>1$, subject to a prescribed upper bound and feasibility. Similar to Bonettini \etal~\cite{Bonettini_2017}, we also add a rule at the end of each iteration to select, as the updated point, the one that yields a smaller value of the objective function between $y^k$ and the point provided from the line search.
We prove that the line search terminates in finitely many steps for a general nonsmooth convex $g$. Under the assumptions stated in \cref{sec:convergence-analysis}, we establish that the accumulation points of a sequence generated by ABPG-VMA are stationary points. Furthermore, by assuming the Kurdyka--\L{}ojasiewicz property~\cite{90074980-9c0c-33c4-b119-70171dab0b45} for $\Psi$, we prove that our algorithm achieves global convergence even for $g\not\equiv 0$.

Moreover, numerical experiments on $\ell_p$ regularized least squares problems, nonnegative linear inverse problems, and phase retrieval demonstrate that ABPG-VMA outperforms ABPG and other existing algorithms. In particular, ABPG-VMA achieves a reduction in the objective function value in fewer iterations than ABPG.

The structure of this paper is as follows.  
\cref{sec:preliminaries} introduces essential notation such as the subdifferential, the Bregman distances, and the Kurdyka--\L{}ojasiewicz property.
In \cref{sec:proposed-algorithm}, we propose ABPG-VMA and discuss its line search condition.
\cref{sec:convergence-analysis} presents properties of ABPG-VMA and establishes its global convergence.
\cref{sec:numerical-experiments} presents numerical experiments on $\ell_p$ regularized least squares problems, nonnegative linear inverse problems, and phase retrieval.
Finally, in \cref{sec:conclusion}, we present conclusions.

\paragraph{Notation}
In this paper, we use the following notation.
Let $\RR$ and $\RR_+$ be the sets of real numbers and nonnegative real numbers, respectively.
Let $\RR^n$ and $\RR_+^n$ be the real space of $n$ dimensions and the nonnegative orthant of $\RR^n$, respectively.
Let $\RR^{n \times m}$ be the set of $n \times m$ real matrices. The identity matrix is $I \in \RR^{n \times n}$.
Let $|x|$ and $x^p$ be the elementwise absolute and $p$th power vectors of
$x \in \RR^n$, respectively. Given a real number $p \geq 1$, the $\ell_p$ norm is defined by $\|x\|_p = \left( \sum_{i=1}^n |x_i|^p \right)^{1/p}$.
Let $\lambda_{\max}(M)$ be the largest eigenvalue of a symmetric matrix $M \in \RR^{n \times n}$.

Let $B(\bar{x},r) = \Set{x \in \RR^n}{\|x - \bar{x}\|\leq r}$ denote the ball with center $\bar{x} \in \RR^n$ and radius $r > 0$.
Let $\interior C$ and $\cl C$ be the interior and the closure of a set $C \subset \RR^n$, respectively. The distance from a point $x \in \RR^n$ to $C$ is defined by $\dist(x, C) := \inf_{y \in C} \|x - y\|$.
The indicator function $\delta_C$ is defined by $\delta_C(x) = 0$ for $x\in C$ and $\delta_C(x) = +\infty$ otherwise.
The sign function $\sgn(x)$ is defined by $\sgn(x) = -1$ for $x < 0$, $\sgn(x) = 0$ for $x = 0$, and $\sgn(x) = 1$ for $x > 0$.

Given $y\in\RR^n$ and $z\in\RR$, we define a set $[\Psi(y)<\Psi<\Psi(y)+z]$ as the set of all $x$ in the subset of $\RR^n$ that satisfy $\Psi(y)<\Psi(x)<\Psi(y)+z$.
Given $k \in \mathbb{N}$, let $\mathcal{C}^k$ be the class of $k$-times continuously differentiable functions.

\section{Preliminaries}\label{sec:preliminaries}
\subsection{Subdifferentials}
First, we introduce the definitions of subdifferentials. For an extended-real-valued function $f:\RR^n \rightarrow [-\infty,+\infty]$, the effective domain of $f$ is defined by $\dom f \coloneqq \Set{x\in\RR^n}{f(x)<+\infty}$. The function $f$ is proper if $f(x)>-\infty$ for all $x\in \RR^n$ and $\dom f \neq \emptyset$.
\begin{definition}[Regular and Limiting Subdifferentials~{\cite[Definition 8.3]{rockafellar2009variational}}]
    Let $f:\RR^n\rightarrow (-\infty,+\infty ]$ be a proper and lower semicontinuous function.
    \begin{enumerate}
        \item The \emph{regular subdifferential} of $f$ at $x\in\dom f$ is defined by
        \begin{align}
            \hat{\partial} f(x) = \left\{ \xi \in \RR^n \,\middle|\, \liminf_{y\to x,y\neq x}\frac{f(y)-f(x)-\langle \xi,y-x\rangle}{\|x-y\|}\geq 0\right\}.
        \end{align}
        When $x\not\in\dom f$, we set $\hat{\partial} f(x)=\emptyset$.
        \item The \emph{limiting subdifferential} of $f$ at $x\in\dom f$ is defined by
        \begin{align}
            \partial f(x) = \left\{ \xi \in \RR^n \,\middle|\, \exists x^k \xrightarrow{f} x,\xi^k \to \xi, \forall k\in\NN, \xi^k \in \hat{\partial} f(x^k)\right\},
        \end{align}
        where $x^k \xrightarrow{f} x$ means $x^k \to x$ and $f(x^k) \to f(x)$.
    \end{enumerate}
\end{definition}
Generally, $\hat{\partial} f(x) \subset \partial f(x)$ holds for all $x\in\RR^n$~\cite[Theorem 8.6]{rockafellar2009variational}. We define $\dom \partial f \coloneqq \{x\in \RR^n| \partial f(x) \neq \emptyset\}$. If $f$ is convex, the regular and limiting subdifferentials coincide with the (classical) subdifferential~\rm{\cite[Proposition 8.12]{rockafellar2009variational}}.

\subsection{Bregman Distances}
Let $C$ be a nonempty and convex subset of $\RR^n$. We introduce the kernel generating distance~\cite{Bolte-2018} and the Bregman distance.
\begin{definition}[Kernel Generating Distances~{\cite[Definition 2.1]{Bolte-2018}}]\label[definition]{kernel-generative}
    A function $\phi:\RR^n\rightarrow(-\infty,+\infty]$ is called a \emph{kernel generating distance} associated with $C$ if it satisfies the following conditions:
    \begin{enumerate}
        \item $\phi$ is a proper, lower semicontinuous, and convex function, with $\dom \phi\subset \cl C$ and $\dom \partial \phi = C$.
        \item $\phi$ is $\mathcal{C}^1$ on $\interior\dom \phi \equiv C$.
    \end{enumerate}
    We denote $\mathcal{G}(C)$ as the class of kernel generating distances associated with $C$.
\end{definition}
\begin{definition}[Bregman Distances~{\cite{Bregman1967TheRM}}]
    For a kernel generating distance $\phi\in\mathcal{G}(C)$, a \emph{Bregman distance} $D_{\phi}:\dom \phi \times \interior\dom \phi \rightarrow \RR_+$ is defined by
    \begin{align}
        D_{\phi}(x,y)=\phi(x) - \phi(y) - \langle \nabla\phi(y),x-y \rangle.
    \end{align}
\end{definition}
Because the Bregman distance does not satisfy the symmetry and the triangle inequality, it is not a distance. Due to the convexity of $\phi$, $D_{\phi}(x,y)\geq 0$ for any $(x,y) \in \dom \phi \times \interior\dom \phi$. If $\phi$ is strictly convex, $D_{\phi}(x,y) = 0$ holds if and only if $x = y$.
We also show some examples of Bregman distances.
\begin{example}\quad
    \begin{itemize}
        \item Mahalanobis distance: Let $\phi(x) = \frac{1}{2}\langle Ax, x\rangle$ for a positive definite matrix $A\in\RR^{n\times n}$ and $\dom\phi = \RR^n$. 
        Then, we have $D_\phi(x,y)=\frac{1}{2}\langle A(x - y), x - y\rangle$, which is called the Mahalanobis distance. When $A = I$, the Mahalanobis distance corresponds with the squared Euclidean distance, \ie, $D_\phi(x,y)=\frac{1}{2}\|x - y\|^2$.
        \item Kullback--Leibler divergence~\textup{\cite{Kullback1951-yv}}: Let $\phi$ be the Boltzmann--Shannon entropy, \ie, $\phi(x) = \sum_{i=1}^nx_i\log x_i$ with $0\log 0 = 0$ and $\dom\phi = \RR_+^n$.
        Then, we have $D_\phi(x,y)=\sum_{i=1}^n (x_i\log \frac{x_i}{y_i} - x_i + y_i)$, which is called the Kullback--Leibler divergence.
        \item Itakura--Saito divergence~\textup{\cite{Itakura1968-en}}: Let $\phi$ be the Burg entropy, \ie, $\phi(x) = -\sum_{i=1}^n\log x_i$ and $\dom\phi = \RR^n_{++}$.
        Then, we have $D_\phi(x,y)=\sum_{i=1}^n \left(\frac{x_i}{y_i} - \log\frac{x_i}{y_i} - 1\right)$, which is called the Itakura--Saito divergence.
    \end{itemize}
\end{example}
See~\cite{Bauschke2017-hg,Bauschke1997-vk,Lu-haihao} and~\cite[Table 2.1]{Dhillon2008-zz} for more examples.

\subsection{Kurdyka--Łojasiewicz Property}
The Kurdyka--\L{}ojasiewicz (KL) property is an essential assumption to establish global convergence. Attouch \etal~\cite{90074980-9c0c-33c4-b119-70171dab0b45} extended the \L{}ojasiewicz gradient inequality~\cite{Kurdyka1998,AIF_1993__43_5_1575_0} to nonsmooth functions.

For $v>0$, we define $\Xi_{v}$ as a set of all continuous concave functions $\psi:[0,v)\rightarrow\RR_+$ that are $\mathcal{C}^1$ on $(0,v)$ and satisfy $\psi(0)=0$, and whose derivative $\psi'(x)$ is positive on $(0,v)$.
We define the Kurdyka--\L{}ojasiewicz property.
\begin{definition}[Kurdyka--\L{}ojasiewicz Property~{\cite{90074980-9c0c-33c4-b119-70171dab0b45}}]\label[definition]{def:KL-property}
    Let $f:\RR^n\rightarrow (-\infty, +\infty]$ be a proper and lower semicontinuous function.
    The function $f$ is said to satisfy the \emph{Kurdyka--\L{}ojasiewicz property} (for short: KL property) at $\bar{x}\in \dom \partial f$ if there exist $v \in (0,+\infty ]$, a neighborhood $U$ of $\bar{x}$, and a function  $\psi\in\Xi_{v}$, such that for any $x \in U \cap [f(\bar{x}) < f < f(\bar{x})+v]$, the following inequality holds:
    \begin{align}
        \psi '(f(x)-f(\bar{x}))\dist(0,\partial f(x))\geq 1. \label{ineq:KL-property}
    \end{align}
    Moreover, $f$ is called a \emph{KL function} if $f$ satisfies the KL property at each point of $\dom\partial f$.
\end{definition}
The uniformized KL property is established by the KL property.
\begin{lemma}[Uniformized KL property~{\cite[Lemma 6]{bolte-2014}}]\label[lemma]{lemma:uniformized-kl}
    Assume that $f:\RR^n\rightarrow (-\infty, +\infty]$ is a proper and lower semicontinuous function. If $f$ takes a constant value on some compact set $\Gamma$, and satisfies the KL property on $\Gamma$, then there exist $v,\epsilon \in (0,+\infty ]$, and $\psi\in\Xi_{v}$ such that, for any $\bar{x}\in \Gamma$, and any $x\in\RR^n$ satisfying $\dist(x,\bar{x})<\epsilon$ and $x\in[f(\bar{x}) < f < f(\bar{x})+v]$, the following inequality holds:
    \begin{align}
        \psi '(f(x)-f(\bar{x}))\dist(0,\partial f(x))\geq 1.
    \end{align}
\end{lemma}

\section{Proposed Algorithm: Approximate Bregman Proximal Gradient Algorithm with Variable Metric Armijo Line Search}\label{sec:proposed-algorithm}
Throughout this paper, we make the following assumptions.
\begin{assumption}\label[assumption]{asmp:function-condition}\quad
    \begin{enumerate}
        \item $\phi\in\mathcal{G}(C)$ with $\cl C = \cl \dom \phi$ is $\mathcal{C}^2$ on $C = \interior\dom \phi$.
        \item $f:\RR^n\rightarrow (-\infty,+\infty]$ is proper and lower semicontinuous with $\dom \phi \subset \dom f$ and is $\mathcal{C}^1$ on $C$.
        \item $g:\RR^n\rightarrow (-\infty,+\infty]$ is proper, lower semicontinuous, and convex with $C \subset \dom g$.
        \item $\Psi^*\coloneqq \inf_{x\in \cl C}\Psi(x)>-\infty$.
        \item For any $x\in\interior\dom\phi$ and $\lambda>0$, $u \mapsto
        \lambda g(u) + \frac{1}{2}\langle \nabla^2 \phi(x)(u-x),u-x\rangle$ is supercoercive, that is,
        \begin{align}
            \lim_{\|u\|\to\infty}\frac{\lambda g(u) + \frac{1}{2}\langle \nabla^2 \phi(x)(u-x),u-x\rangle}{\|u\|} = \infty.
        \end{align}
    \end{enumerate}
\end{assumption}
\cref{asmp:function-condition}(i)-(iv) are standard assumptions for Bregman-type algorithms~\cite{Bolte-2018,Takahashi} and are generally satisfied in practice.
For any $x \in C$, $\partial (g + \delta_{\cl{C}})(x) = \partial g (x) + \partial\delta_{\cl{C}}(x) = \partial g(x)$ holds because $x$ is an interior point of $C$ and $\dom g$ from \cref{asmp:function-condition}(iii) and $\partial\delta_{\cl{C}}(x) = \{0\}$.
For example, \cref{asmp:function-condition}(v) holds if $\phi$ is strongly convex. Note that we will assume the strong convexity of $\phi$ in \cref{asmp:strongly-convex}.

\subsection{Approximate Bregman Proximal Gradient Algorithm}
Let $\phi\in\mathcal{G}(C)$ be $\mathcal{C}^2$ on $C$. Takahashi and Takeda~\cite{Takahashi} define the approximate Bregman distance $\tilde{D}_{\phi}(u,x)\geq 0$, using a second-order approximation of $\phi(u)$ for $u \in\dom\phi$ around point $x \in \interior\dom\phi$, as
\begin{align}
    \tilde{D}_{\phi}(u,x) \coloneqq \frac{1}{2}\langle \nabla^2 \phi(x)(u-x),u-x\rangle\simeq D_{\phi}(u,x). \label{App-Breg-dist}
\end{align}
Note that $D_{\phi}(u,x)\leq \tilde{D}_{\phi}(u,x)$ or $D_{\phi}(u,x)\geq \tilde{D}_{\phi}(u,x)$ does not necessarily hold for any $x$ and $u$. Therefore, a line search was incorporated into the proposed algorithm.

The Bregman proximal gradient mapping~\cite{Bolte-2018} at a point $x\in C$ for a parameter $\lambda > 0$ is defined by
\begin{align}
    \mathcal{T}_{\lambda}(x)\coloneqq \argmin_{u\in\cl C}\left\{\langle \nabla f(x), u- x \rangle + g(u) + \frac{1}{\lambda}D_{\phi}(u,x)\right\}.\label{prox-Bregman}
\end{align}
Instead of \cref{prox-Bregman}, the approximate Bregman proximal gradient mapping~\cite{Takahashi} at a point $x\in C$ is defined by
\begin{align}
    \tilde{\mathcal{T}}_{\lambda}(x)\coloneqq \argmin_{u\in\cl C}\left\{\langle \nabla f(x), u- x \rangle + g(u) + \frac{1}{\lambda}\tilde{D}_{\phi}(u,x)\right\}.\label{prox-Ap-Bregman}
\end{align}
Using \cref{asmp:function-condition}(iii) and the positive semidefiniteness of $\nabla^2\phi$, \cref{prox-Ap-Bregman} is a convex optimization problem.
\begin{assumption}\label[assumption]{asmp:feasibility}
    For any $x\in C$ and any $\lambda >0$, $\tilde{\mathcal{T}}_{\lambda}(x)\subset C$ holds.
\end{assumption}
\cref{asmp:feasibility} ensures that the subproblem solutions belong to $C$.
Obviously, when $C \equiv \RR^n$, \cref{asmp:feasibility} holds.
By \cref{asmp:function-condition}(iii) and (v), the objective
in \cref{prox-Ap-Bregman} is proper, lower semicontinuous,
and coercive on $\cl C$.
We have the following well-posedness result.
\begin{lemma}[Well-posedness of $\tilde{\mathcal{T}}_{\lambda}$~{\cite[Lemma 12]{Takahashi}}]\label[lemma]{lem:feasibility}
    Suppose that \cref{asmp:function-condition,asmp:feasibility} hold. For any $x\in\interior\dom\phi$ and any $\lambda>0$, the approximate Bregman proximal gradient mapping $\tilde{\mathcal{T}}_{\lambda}(x)$ is a nonempty compact subset of $C$.
\end{lemma}

\subsection{Variable Metric Armijo Line Search}
Since $\Psi(y^k)\leq\Psi(x^k)$ is not necessarily guaranteed for the solution of the subproblem $y^k\in\tilde{\mathcal{T}}_{\lambda}(x^k)$ with any $\lambda > 0$, Takahashi and Takeda~\cite{Takahashi} introduced the line search procedure for ABPG. To ensure global convergence, we improved the condition of the line search procedure. The new condition is inspired by the line search method based on the Armijo--Wolfe condition proposed by Lewis and Overton~\cite{Lewis2013} and Chao \etal~\cite{Miantao2024}, and on the Armijo-like line search method introduced by Bonettini \etal~\cite{Bonettini_2017}. We also execute an update step to take a point corresponding to a smaller value of the objective function at the end of each iteration, inspired by Bonettini \etal~\cite{Bonettini_2017}.

In order to define the search direction $d^k = y^k - x^k$, we solve the subproblem $y^{k} \in \tilde{\mathcal{T}}_{\lambda}(x^k)$.
Let $0<c_1<1$ and fix an upper bound $1\leq t_{\max}<\infty$. To ensure that $\Psi(x^{k+1})$ is sufficiently smaller than $\Psi(x^k)$, we impose the following condition on $t\in(0,t_{\max}]$ with $x^k+td^k\in C$:
\begin{align}
    &\Psi(x^{k} + t d^{k})\\ &\qquad< \Psi(x^{k}) + c_1 t \left(\langle \nabla f(x^{k}), d^{k}\rangle +g(x^k+d^k) - g(x^k) + \frac{1}{2\lambda}\langle \nabla^2 \phi (x^k)d^{k}, d^{k} \rangle\right). \label{condition1}
\end{align}
Since $x^k,y^k\in C$ and $C$ is convex, $x^k+td^k=(1-t)x^k+ty^k\in C$ for $t\in[0,1]$.
Here, by rearranging~\cref{condition1}, we define
\begin{align}
    A_k(t) &\coloneqq \Psi(x^{k} + t d^{k}) - \Psi(x^{k})\\
    &\quad- c_1 t \left(\langle \nabla f(x^{k}), d^{k}\rangle + g(x^k+d^k) - g(x^k) + \frac{1}{2\lambda}\langle \nabla^2 \phi (x^k)d^{k}, d^{k} \rangle\right). \label{ccondition1}
\end{align}
For $x^k+td^k\notin C$, we set $A_k(t)\coloneqq+\infty$ without evaluating $\Psi$.
The line search condition~\cref{condition1}, together with feasibility, can be rewritten as $A_k(t) < 0$.

Now we describe the proposed algorithm and its line search procedure for solving~\cref{problem}.
The subproblem $\tilde{\mathcal{T}}_{\lambda}(x^k)$ on \cref{line:subproblem} is convex, and it is strongly convex if $\phi$ is strongly convex (see also \cref{asmp:strongly-convex}). On \cref{line:step-size}, we obtain $t_k$ using \cref{alg:vmw}.
We use a contraction factor $\eta\in(0,1)$ and an expansion factor $\delta>1$. If $A_k(1)\geq0$, we backtrack from $t=1$, multiplying $t$ by $\eta$ until $A_k(t)<0$. Otherwise, we increase $t$ to $\min\{\delta t,t_{\max}\}$ while the trial is accepted, and return the last accepted step. The choice $t_{\max}=1$ recovers the backtracking procedure without expansion.

\begin{remark}[Comparison with existing line search conditions]\label[remark]{remark:line-search}
    Our variable metric line search condition~\cref{condition1} allows a larger $t$ than  the classical Armijo condition:
    \begin{align}
        &\Psi(x^{k} + t d^{k})\\
        &\qquad < \Psi(x^{k}) + c_1 t(\langle \nabla f(x^{k}), d^{k}\rangle + g(x^k + d^k) - g(x^k))\\
        &\qquad \leq \Psi(x^{k}) + c_1 t \left(\langle \nabla f(x^{k}), d^{k}\rangle +g(x^k+d^k) - g(x^k) + \frac{1}{2\lambda}\langle \nabla^2 \phi (x^k)d^{k}, d^{k} \rangle\right),
    \end{align}
    where the second inequality follows from the variable metric $\frac{1}{2\lambda}\langle \nabla^2 \phi (x^k)d^{k}, d^{k} \rangle \geq 0$.
    When $\phi$ is strongly convex, the above second inequality holds with strict inequality for $d^k \neq 0$.
    Moreover, the parameter $\lambda$ can be chosen freely.
    As a practical guideline for choosing $\lambda$, we refer the reader to~\cref{remark:l-smad}.
\end{remark}

Although we can choose any $\lambda > 0$, it is better to use $\lambda < 1/L$ for some $L$ as follows (see, for specific examples, \cref{sec:numerical-experiments}).
\begin{remark}\label[remark]{remark:l-smad}
    As already mentioned in~\cref{remark:line-search}, the parameter $\lambda$ can be chosen to be any positive scalar. We now provide a practical guideline for selecting $\lambda$ to reduce the number of iterations.
    In practice, it is better to choose $\lambda < 1 / L$, where $L > 0$ is a parameter given by the smooth adaptable property, \ie, the pair $(f,\phi)$ is said to be $L$-\emph{smooth adaptable} (for short: $L$-\emph{smad})~\cite{Bolte-2018} if there exists $L>0$ such that both $L\phi -f$ and $L\phi+f$ are convex on $C$.
    The $L$-smad property provides the first-order approximation of $f$ by its descent lemma~\cite[Lemma 2.1]{Bolte-2018}, which may reduce the number of iterations (see~\cref{fig:l-smad-varying} for varying $L$).
    Moreover, when $f$ and $\phi$ are $\mathcal{C}^2$, the pair $(f, \phi)$ is $L$-smad if and only if $-L \nabla^2 \phi(x)\preceq\nabla^2 f(x) \preceq L \nabla^2 \phi(x)$ holds for any $x \in C$. In order to achieve superior performance, it is recommended to choose a smaller $L$ and a $\phi$ that shares a similar structure with $f$. See, for more examples of the $L$-smad property,~\cite[Lemma 5.1]{Bolte-2018},~\cite[Lemmas 7 and 8]{Bauschke2017-hg}, \cite[Proposition 24]{Takahashi}, \cite[Proposition 2.1]{Mukkamala2019-mk}, \cite[Theorem 4.1]{Takahashi2026-rv}, \cite[Theorem 1]{Takahashi2023-uh}, and \cite[Propositions 2.1 and 2.3]{Dragomir2021-rv}.
\end{remark}

\begin{algorithm}[!tp]
    \caption{Approximate Bregman proximal gradient algorithm with variable metric Armijo line search (ABPG-VMA)}
    \DontPrintSemicolon
    \label{alg:ABPG-VMA}
    \KwIn{$x^0\in C$, $0<c_1<1$, $\lambda > 0$, $0<\eta<1$, $\delta > 1$, and $1\leq t_{\max}<\infty$}

    \For{$k = 0, 1, 2, \dots$}{
        $y^{k} \gets \tilde{\mathcal{T}}_{\lambda}(x^k)$\; \label{line:subproblem}
        $d^{k} \gets y^{k} - x^{k}$\;
        \If{$d^k=0$}{\Return{$x^k$}\;}
        Compute $t_k$ using $\Linesearch_k(x^k, d^k, c_1, \lambda, \eta, \delta, t_{\max})$ in~\cref{alg:vmw}.\; \label{line:step-size}

        $x^{k+1} \gets \left\{\begin{array}{ll}
             y^{k}&\text{if} \quad \Psi(y^k)<\Psi(x^{k} + t_k d^{k}),  \\
             x^{k} + t_k d^{k}& \text{otherwise}.
        \end{array}\right.$\;\label{line:ABPG-VMA-min-obj}
    }
\end{algorithm}
\begin{algorithm}[!t]
\caption{Variable metric Armijo line search}
\label{alg:vmw}
\DontPrintSemicolon
\KwOut{Step-size $t_k$}
\Procedure{$\Linesearch_k(x, d, c_1, \lambda, \eta, \delta, t_{\max})$}{
$t\gets1$\;
\eIf{$A_k(1)<0$}{
    \While{$t<t_{\max}$}{
        $t_+\gets\min\{\delta t,t_{\max}\}$\;
        \If{$x+t_+d\notin C$ or $A_k(t_+)\geq0$}{\Return{$t$}\;}
        $t\gets t_+$\;
    }
}{
    \While{$A_k(t)\geq0$}{
        $t\gets\eta t$\;
    }
}
\Return{$t$}\;
}
\end{algorithm}
In the next section, we demonstrate that the search direction and step-size in the line search are well-defined and that the sequence of points generated by ABPG-VMA globally converges to a stationary point.

\section{Convergence Analysis}\label{sec:convergence-analysis}
Throughout this section, we make the following assumption.
\begin{assumption}\label[assumption]{asmp:strongly-convex}
    For a positive number $\sigma>0$, $\phi$ is $\sigma$-strongly convex on $C$.
\end{assumption}
Under \cref{asmp:strongly-convex}, since the objective function defining $\tilde{\mathcal{T}}_{\lambda}(x)$ is strongly convex and lower semicontinuous, it has a unique minimizer.
If $\phi$ is convex but not strongly convex, one can enforce strong convexity by adding the quadratic term $\frac{1}{2}\|\cdot\|^2$. When $g \equiv 0$, this modification does not affect the difficulty of the subproblems. When $g \not\equiv 0$, we are merely adding a quadratic and separable term $\frac{1}{2}\|\cdot\|^2$ to the subproblem that already contains a quadratic term. Therefore, the subproblems remain relatively easy to solve.

\subsection{Properties of Proposed Algorithm}
We first show the search direction property. More precisely, we prove that $d$ is a descent direction. The following inequality is a modified version of~\cite[Proposition 15]{Takahashi}.
\begin{proposition}[Search direction property]\label[proposition]{Search_direction_property}
    Suppose that \cref{asmp:function-condition,asmp:feasibility,asmp:strongly-convex} hold. For any $x \in\interior\dom\phi$, let $\xi\in\partial g(x)$. For any $\lambda > 0$ and $d = y - x\neq0$ defined by
    \begin{align}
        y = \tilde{\mathcal{T}}_{\lambda}(x)\label{subproblem-prop}
    \end{align}
    we have
    \begin{align}
        \langle\nabla f(x)+\xi, d\rangle\leq\langle\nabla f(x), d\rangle + g(x+d) - g(x) \leq -\frac{1}{\lambda}\langle\nabla^2\phi(x)d,d\rangle < 0. \label{ineq:search-direction}
    \end{align}
\end{proposition}
\begin{proof}
    Since $g$ is convex, we have
    \begin{align}
        \langle \xi, y - x \rangle \leq g(y) - g(x),
    \end{align}
    which implies
    \begin{align}
        \langle\nabla f(x)+\xi, d\rangle\leq\langle\nabla f(x), d\rangle + g(x+d) - g(x).
    \end{align}
    From the first-order optimality condition of~\cref{subproblem-prop}, we have
    \begin{align}
       -\nabla f(x) - \frac{1}{\lambda}\nabla^2\phi(x)(y - x) \in \partial( g+\delta_{\cl{C}})(y). \label{cond:first-optimality-subproblem}
    \end{align}
    Since $g$ is convex and $\delta_{\cl{C}}(x) = \delta_{\cl{C}}(y) = 0$ from \cref{lem:feasibility}, it holds that
    \begin{align}
        g(y) - g(x) \leq -\left\langle\nabla f(x) + \frac{1}{\lambda}\nabla^2\phi(x)(y - x),y - x\right\rangle.
    \end{align}
    Therefore, substituting $y \leftarrow x + d$ into the above inequality, we obtain
    \begin{align}
        \langle\nabla f(x), d\rangle + g(x + d) - g(x) &\leq \langle\nabla f(x), d\rangle - \left\langle\nabla f(x)+\frac{1}{\lambda}\nabla^2\phi(x)d,d\right\rangle\\
        &=-\frac{1}{\lambda}\langle\nabla^2\phi(x)d,d\rangle < 0,
    \end{align}
    where the last inequality holds because $\phi$ is strongly convex.
\end{proof}
When $t$ satisfies \cref{condition1}, we guarantee that the objective function value decreases.
\begin{lemma}[Sufficient decrease property]\label[lemma]{lem:value-reduction}
    Suppose that \cref{asmp:function-condition,asmp:feasibility,asmp:strongly-convex} hold. For any $\lambda > 0$, $x \in\interior\dom\phi$ and $d = y - x$ defined by~\cref{subproblem-prop}, let $t>0$ satisfy $x+td\in C$ and~\cref{condition1}. Then, the following inequality holds:
    \begin{align}
        \Psi(x^+) - \Psi(x) \leq -\frac{c_1 t}{2\lambda}\langle\nabla^2\phi(x)d,d\rangle \leq 0, \label{ineq:sufficient-decrease}
    \end{align}
    where
    \begin{align}
        x^+ = \left\{\begin{array}{ll}
             y,&\text{if} \quad \Psi(y)<\Psi(x + t d), \\
             x + t d,& \text{otherwise}.
        \end{array}\right. \label{ineq:obj-decrease}
    \end{align}
\end{lemma}
\begin{proof}
    From $y = x + d$, we have
    \begin{align}
        \Psi(x^+) - \Psi(x)&\leq \Psi(x+td) - \Psi(x) \\
        &< c_1 t \left(\langle \nabla f(x),d\rangle + g(x+d) - g(x) + \frac{1}{2\lambda}\langle \nabla^2 \phi (x)d,d  \rangle \right)\\
        &\leq -\frac{c_1 t}{2\lambda}\langle\nabla^2\phi(x)d,d\rangle \leq 0,
    \end{align}
    where the first inequality holds from~\cref{ineq:obj-decrease}, the second inequality holds from~\cref{condition1}, and the third inequality holds from~\cref{ineq:search-direction}.
\end{proof}
The above lemma indicates that the objective function value is reduced at every step.

\begin{corollary}[Well-definedness of the line search]\label[corollary]{cor:well-defined}
    Suppose that \cref{asmp:function-condition,asmp:feasibility,asmp:strongly-convex} hold.
    At every nonterminal iteration, \cref{alg:vmw} terminates finitely and returns
    $t_k\in(0,t_{\max}]$ satisfying $A_k(t_k)<0$ and $x^k+t_kd^k\in C$.
\end{corollary}
\begin{proof}
    Since $x^k,y^k\in C$ and $C$ is convex, $x^k+td^k\in C$ for $t\in[0,1]$.
    Convexity of $g$ gives
    $g(x^k+td^k)-g(x^k)\leq t(g(y^k)-g(x^k))$.
    Combining this with the fundamental theorem of calculus
    and~\cref{ineq:search-direction}, we obtain
    \begin{align}
        A_k(t)&\leq -\frac{(2-c_1)t}{2\lambda}
        \langle\nabla^2\phi(x^k)d^k,d^k\rangle+t\int_0^1\langle\nabla f(x^k+std^k)-\nabla f(x^k),d^k\rangle\,ds.
        \label{ineq:basic-estimate}
    \end{align}
    For fixed $k$, the integral tends to zero as $t\downarrow0$.
    Since $d^k\neq0$ and $\nabla^2\phi(x^k)\succeq\sigma I$, we have
    $A_k(t)<0$ for all sufficiently small $t>0$, so the backtracking branch terminates.
    If $A_k(1)<0$, repeated expansion $t\gets\min\{\delta t,t_{\max}\}$ reaches $t_{\max}$ in finitely many steps unless a trial is rejected earlier, since $\delta>1$.
    The last accepted step is returned in either case, and $A_k(t_k)<0$ also ensures $x^k+t_kd^k\in C$.
\end{proof}

\subsection{Global Subsequential Convergence}
In this subsection, we discuss global subsequence convergence. In other words, we show that any accumulation point of a sequence $\{x^k\}_{k\in\NN}$ generated by ABPG-VMA is a stationary point of~\cref{problem}.
We use the limiting subdifferential and define the stationary point, inspired by Fermat's rule~\cite[Theorem 10.1]{rockafellar2009variational}.
\begin{definition}
    A point $x^*\in C$ is called a stationary point of $\Psi$ if
    \begin{align}
        0\in\nabla f(x^*) + \partial g(x^*).
    \end{align}
\end{definition}
If $d^k=0$, the optimality condition~\cref{cond:first-optimality-subproblem}
shows that the algorithm returns a stationary point. We then extend the
sequence constantly, with $y^k=x^k$, $d^k=0$, and $t_k=1$ thereafter.
\begin{assumption}\label[assumption]{asmp:global-subsequential}\quad
    \begin{enumerate}
        \item The objective function $\Psi$ is level-bounded, \ie, for any $r\in\RR$,
        the set $\Set{x\in\RR^n}{\Psi(x)\leq r}$ is bounded.
        \item $\nabla f$ is Lipschitz continuous on any compact subset of $C$.
        \item Any accumulation points of $\{x^k\}_{k\in\NN}$ belong to $C$.
    \end{enumerate}
\end{assumption}
\cref{asmp:global-subsequential}(i) and (ii) are often assumed in nonsmooth optimization when the problem includes nonsmooth lower semicontinuous functions~\cite{Bolte-2018,Takahashi}. In fact, a lower semicontinuous, level-bounded, and proper function has a minimum~\cite[Theorem 1.9]{rockafellar2009variational}. \cref{asmp:global-subsequential}(iii) is also assumed in~\cite[Theorem 4.1(ii)]{Ochs2019-ad} and~\cite[Theorems 17 and 18]{Mukkamala2022-jo}.
We have the following global subsequential convergence result.

\begin{theorem}[Global subsequential convergence]\label[theorem]{thm:global-subsequential-convergence}
    Suppose that \cref{asmp:function-condition,asmp:feasibility,asmp:strongly-convex,asmp:global-subsequential} hold.
    Let $\{x^k\}_{k\in\NN}$ be a sequence generated by ABPG-VMA. Then, the following statements hold:
    \begin{enumerate}
        \item The sequence $\{x^k\}_{k\in\NN}$ is bounded.
        \item It holds that $0<t_{\min}\coloneqq \eta\min\left\{1,\frac{(2-c_1)\sigma}{\lambda M_1}\right\}\leq t_k\leq t_{\max}$, 
        where $M_1>0$ is a Lipschitz constant of $\nabla f$ on a compact convex set $\mathcal S\subset C$ containing all $x^k,y^k$.
        Moreover, $M_2\coloneqq\max_{x\in\mathcal S}\|\nabla^2\phi(x)\|<\infty$.
        \item $\lim_{k \to \infty}\|d^k\| = 0$ and $\lim_{k \to \infty}\|x^{k+1}-x^k\| = 0$ hold.
        \item Any accumulation point of $\{x^k\}_{k\in\NN}$ is a stationary point of~\cref{problem}.
    \end{enumerate}
\end{theorem}
\begin{proof}
    (i) Since $\Psi(x^{k+1}) \leq \Psi(x^k)$ from \cref{lem:value-reduction} and $\Psi$ is level-bounded, the sequence of points $\{x^k\}_{k\in\NN}$ is bounded.

    (ii) By (i) and \cref{asmp:global-subsequential}(iii),
    $\mathcal K\coloneqq\overline{\Set{x^k}{k\in\NN}}$ is a compact subset of $C$.
    Take any convergent subsequence $x^{k_j}\to\bar x\in\mathcal K$ and set
    $\bar y=\tilde{\mathcal T}_\lambda(\bar x)\in C$.
    Applying~\cref{cond:first-optimality-subproblem} at $x^{k_j}$ and $\bar x$,
    and using the monotonicity of $\partial(g+\delta_{\cl C})$, we obtain,
    after multiplying by $\lambda>0$,
    \begin{align}
        0&\leq\langle-\lambda(\nabla f(x^{k_j})-\nabla f(\bar x))
        -\nabla^2\phi(x^{k_j})(y^{k_j}-x^{k_j})+\nabla^2\phi(\bar x)(\bar y-\bar x),y^{k_j}-\bar y\rangle.
    \end{align}
    Rearranging this inequality and using $\nabla^2\phi(x^{k_j})\succeq\sigma I$ yields
    \begin{align}
        \sigma\|y^{k_j}-\bar y\|^2
        &\leq\langle\nabla^2\phi(x^{k_j})(y^{k_j}-\bar y),\,y^{k_j}-\bar y\rangle\notag\\
        &\leq\langle\nabla^2\phi(x^{k_j})(x^{k_j}-\bar y)
        -\nabla^2\phi(\bar x)(\bar x-\bar y)-\lambda(\nabla f(x^{k_j})-\nabla f(\bar x)),\,y^{k_j}-\bar y\rangle\\
        &\leq\|\nabla^2\phi(x^{k_j})(x^{k_j}-\bar y)
        -\nabla^2\phi(\bar x)(\bar x-\bar y)-\lambda(\nabla f(x^{k_j})-\nabla f(\bar x))\|\|y^{k_j}-\bar y\|.
    \end{align}
    If $y^{k_j}\neq\bar y$, we divide by $\|y^{k_j}-\bar y\|$ to obtain
    \begin{align}
        \sigma\|y^{k_j}-\bar y\|
        &\leq\|\nabla^2\phi(x^{k_j})(x^{k_j}-\bar y)
        -\nabla^2\phi(\bar x)(\bar x-\bar y)-\lambda(\nabla f(x^{k_j})-\nabla f(\bar x))\|.
    \end{align}
    This estimate also holds when $y^{k_j}=\bar y$.
    Since $x^{k_j}\to\bar x$ and both $\nabla^2\phi$ and $\nabla f$ are continuous on $C$,
    the right-hand side tends to zero, and hence $y^{k_j}\to\bar y$.
    Thus every subsequence of $\{y^k\}$ has a further subsequence converging in $C$.
    Hence, $\mathcal S\coloneqq\operatorname{conv}(\mathcal K\cup\overline{\Set{y^k}{k\in\NN}})$ is a compact convex subset of $C$ containing all $x^k,y^k$.

    For $d^k\neq0$ and $0<t\leq1$, Lipschitz continuity in~\cref{ineq:basic-estimate} gives
    \begin{align}
        A_k(t)\leq\frac{t\|d^k\|^2}{2}
        \left(M_1t-\frac{(2-c_1)\sigma}{\lambda}\right).
        \label{ineq:uniform-bracket}
    \end{align}
    Hence $A_k(t)<0$ whenever $0<t<\hat t\coloneqq
    \min\{1,(2-c_1)\sigma/(\lambda M_1)\}$.
    If $A_k(1)<0$, the algorithm returns $t_k\geq1$.
    Otherwise, the trial preceding the accepted backtracking step is rejected,
    so $t_k/\eta\geq\hat t$ and $t_k\geq\eta\hat t$.
    By construction, $t_k\leq t_{\max}$. The same bounds hold for the constant extension $t_k=1$.

    (iii) Set $\rho_1\coloneqq\frac{c_1\sigma}{2\lambda}\min\{t_{\min},t_{\max}^{-1}\}>0$.
    By the update rule, $\|x^{k+1}-x^k\|\leq\max\{1,t_k\}\|d^k\|$.
    Since (ii) gives $\min\{t_{\min},t_{\max}^{-1}\}\max\{1,t_k^2\}\leq t_k$, \cref{lem:value-reduction} yields
    \begin{align}
        \rho_1\|d^k\|^2\leq\Psi(x^k)-\Psi(x^{k+1}),\qquad
        \rho_1\|x^{k+1}-x^k\|^2\leq\Psi(x^k)-\Psi(x^{k+1}).
        \label{C1}
    \end{align}
    Summing and using $\Psi\geq\Psi^*>-\infty$ proves both limits.

    (iv) Let $x^{k_j}\to\bar x\in C$. By the argument in (ii),
    $y^{k_j}\to\tilde{\mathcal T}_\lambda(\bar x)$, whereas (iii) gives
    $y^{k_j}=x^{k_j}+d^{k_j}\to\bar x$.
    Thus $\tilde{\mathcal T}_\lambda(\bar x)=\bar x$, and
    \cref{cond:first-optimality-subproblem} implies $0\in\nabla f(\bar x)+\partial g(\bar x)$.
\end{proof}

\subsection{Global Convergence}
Now, we show the global convergence of ABPG-VMA.
Before discussing global convergence, we have the following lemma.
\begin{lemma}\label[lemma]{lem:gradient-obj-value}
    Suppose that \cref{asmp:function-condition,asmp:feasibility,asmp:strongly-convex,asmp:global-subsequential} hold.
    Let $\{x^k\}_{k\in\NN}$ be a sequence generated by ABPG-VMA. Then, the following statements hold:
    \begin{enumerate}
        \item There exist $\rho_2>0$ and $w^{k} \in \nabla f(y^k) + \partial (g+\delta_{\cl{C}})(y^k)$ such that
        \begin{align}
            \|w^k\|\leq\rho_2\|x^{k+1}-x^k\|,\qquad
            \rho_2\coloneqq\frac{M_1+M_2/\lambda}{t_{\min}}.
        \end{align}
        \item $\Psi \equiv \zeta$ on $\Omega$, where $\Omega$ is the set of accumulation points of $\{x^k\}_{k\in\NN}$. Moreover, $\lim_{k\to\infty}\Psi(x^k) = \lim_{k\to\infty}\Psi(y^k) = \Psi(\bar{x})$ for any $\bar{x}\in\Omega$.
    \end{enumerate}
\end{lemma}
\begin{proof}
    (i) By~\cref{cond:first-optimality-subproblem}, we may take
    $w^k=\nabla f(y^k)-\nabla f(x^k)-\lambda^{-1}\nabla^2\phi(x^k)d^k$.
    Using the triangle inequality, we obtain
    \begin{align}
        \|w^k\|
        &\leq\|\nabla f(y^k)-\nabla f(x^k)\|
        +\lambda^{-1}\|\nabla^2\phi(x^k)d^k\|\notag\\*
        &\leq M_1\|y^k-x^k\|+(M_2/\lambda)\|d^k\|\notag\\*
        &=(M_1+M_2/\lambda)\|d^k\|,
    \end{align}
    where the second inequality uses the Lipschitz continuity of $\nabla f$
    with constant $M_1$ and the bound $\|\nabla^2\phi(x^k)\|\leq M_2$
    from \cref{thm:global-subsequential-convergence}(ii),
    and the last equality uses $d^k=y^k-x^k$.
    Next, the definition of $t_{\min}$ and $0<\eta<1$ give $0<t_{\min}\leq\eta<1$.
    If $x^{k+1}=y^k$, then
    \begin{align}
        \|x^{k+1}-x^k\|=\|d^k\|\geq t_{\min}\|d^k\|,
    \end{align}
    where the inequality follows from $t_{\min}<1$.
    If $x^{k+1}=x^k+t_kd^k$, then
    \begin{align}
        \|x^{k+1}-x^k\|=t_k\|d^k\|\geq t_{\min}\|d^k\|,
    \end{align}
    where the inequality follows from $t_k\geq t_{\min}>0$.
    Thus, in either case, $\|d^k\|\leq t_{\min}^{-1}\|x^{k+1}-x^k\|$.
    Substituting this bound into the estimate for $\|w^k\|$ yields
    \begin{align}
        \|w^k\|\leq\frac{M_1+M_2/\lambda}{t_{\min}}\|x^{k+1}-x^k\|
        =\rho_2\|x^{k+1}-x^k\|.
    \end{align}

    (ii) Sufficient decrease and $\Psi\geq\Psi^*$ imply
    $\Psi(x^k)\to\zeta$ for some $\zeta\in\RR$.
    Since $g$ is convex and finite on the open set $C$, both $g$ and $\Psi$
    are continuous on $C$. Thus $\Psi\equiv\zeta$ on $\Omega$.
    Uniform continuity of $\Psi$ on the compact set $\mathcal S$, together with
    $\|y^k-x^k\|\to0$, gives $\Psi(y^k)-\Psi(x^k)\to0$.
\end{proof}

\begin{theorem}[Global convergence]\label[theorem]{thm:point-convergence}
    Suppose that \cref{asmp:function-condition,asmp:feasibility,asmp:strongly-convex,asmp:global-subsequential} hold. Furthermore, suppose that $\Psi$ is a KL function. Let $\{x^k\}_{k\in\NN}$ be a sequence generated by ABPG-VMA.
    Then, the sequence $\{x^k\}_{k\in\NN}$ converges to a stationary point of~\cref{problem}; moreover, $\sum_{k = 0}^{\infty}\|x^{k+1}-x^k\| < \infty$.
\end{theorem}
\begin{proof}
    We adapt the proof of~\cite[Theorem 6.2]{Bolte-2018}.
    The claim is immediate if the algorithm terminates.
    Otherwise, \cref{lem:value-reduction} implies that
    $\Psi(x^k)-\zeta>0$ and $\|x^{k+1}-x^k\|>0$ for all $k$, where $\zeta$ is given by
    \cref{lem:gradient-obj-value}(ii).

    Boundedness of $\{x^k\}_{k\in\NN}$ implies that $\Omega$ is nonempty and compact,
    and $\dist(x^k,\Omega)\to0$. Thus $\dist(y^k,\Omega)\to0$ by \cref{thm:global-subsequential-convergence}(iii).
    By \cref{lem:gradient-obj-value}(ii), $\Psi\equiv\zeta$ on $\Omega$.
    Moreover, every point of $\Omega$ is stationary by \cref{thm:global-subsequential-convergence}(iv),
    so $\Psi$ satisfies the KL property on $\Omega$.
    Applying \cref{lemma:uniformized-kl} with $f=\Psi$ and $\Gamma=\Omega$ gives
    $\epsilon,v>0$ and $\psi\in\Xi_v$.
    Since $\Psi(y^{k-1})\to\zeta$ and $\Psi(x^k)\leq\Psi(y^{k-1})$,
    there exists $N\geq1$ such that, for all $k\geq N$,
    \begin{align}
        \dist(y^{k-1},\Omega)<\epsilon,\qquad
        0<\Psi(x^k)-\zeta\leq\Psi(y^{k-1})-\zeta<v.
    \end{align}
    Hence the uniformized KL inequality applies at $y^{k-1}$.
    Since $w^{k-1}\in\partial\Psi(y^{k-1})$ and $\psi'$ is nonincreasing,
    \cref{lem:gradient-obj-value}(i) yields
    \begin{align}
        1\leq\psi'(\Psi(y^{k-1})-\zeta)\|w^{k-1}\|\leq\rho_2\psi'(\Psi(x^k)-\zeta)\|x^k-x^{k-1}\|.
        \label{ineq:dev-psi}
    \end{align}
    Together with concavity of $\psi$ and~\cref{C1}, this yields
    \begin{align}
        \psi(\Psi(x^k)-\zeta)-\psi(\Psi(x^{k+1})-\zeta)
        &\geq\psi'(\Psi(x^k)-\zeta)(\Psi(x^k)-\Psi(x^{k+1}))\geq\frac{\rho_1\|x^{k+1}-x^k\|^2}{\rho_2\|x^k-x^{k-1}\|}.
    \end{align}
    Rearranging the last inequality gives
    \begin{align}
        \|x^{k+1}-x^k\|^2
        \leq\frac{\rho_2}{\rho_1}\|x^k-x^{k-1}\|
        [\psi(\Psi(x^k)-\zeta)-\psi(\Psi(x^{k+1})-\zeta)].
    \end{align}
    The difference in brackets is nonnegative because $\Psi(x^{k+1})\leq\Psi(x^k)$
    and $\psi$ is increasing. Hence,
    \begin{align}
        2\|x^{k+1}-x^k\|
        &\leq2\sqrt{\|x^k-x^{k-1}\|\frac{\rho_2}{\rho_1}
        [\psi(\Psi(x^k)-\zeta)-\psi(\Psi(x^{k+1})-\zeta)]}\notag\\*
        &\leq\|x^k-x^{k-1}\|+\frac{\rho_2}{\rho_1}
        [\psi(\Psi(x^k)-\zeta)-\psi(\Psi(x^{k+1})-\zeta)].
    \end{align}
    Here, the first inequality follows by taking square roots of the preceding bound.
    The second is the arithmetic--geometric mean inequality applied to the
    nonnegative quantities $\|x^k-x^{k-1}\|$ and $\rho_2/\rho_1$ times the
    difference in brackets.
    Summing from $k=N$ to $m\geq N$ gives
    \begin{align}
        &\sum_{k=N}^{m}\|x^{k+1}-x^k\|+\|x^{m+1}-x^m\|\nonumber\\
        &\qquad\leq\|x^N-x^{N-1}\|+\frac{\rho_2}{\rho_1}[\psi(\Psi(x^N)-\zeta)-\psi(\Psi(x^{m+1})-\zeta)].
        \label{ineq:finite-length}
    \end{align}
    Since both $\|x^{m+1}-x^m\|$ and $\psi(\Psi(x^{m+1})-\zeta)$ are nonnegative,
    we obtain, for $m\geq N$,
    \begin{align}
        \sum_{k=N}^{m}\|x^{k+1}-x^k\|
        \leq\|x^N-x^{N-1}\|+\frac{\rho_2}{\rho_1}\psi(\Psi(x^N)-\zeta).
    \end{align}
    The right-hand side is independent of $m$.
    Letting $m\to\infty$ therefore gives $\sum_{k=N}^{\infty}\|x^{k+1}-x^k\|<\infty$.
    Thus the sequence has finite length, is Cauchy, and converges. Its limit is stationary
    by \cref{thm:global-subsequential-convergence}.
\end{proof}

Finally, we establish convergence rates, which are derived from~\cref{C1}, \cref{ineq:dev-psi}, and~\cref{ineq:finite-length} in the same way as, \eg, \cite[Theorem 3]{Bonettini_2017}, \cite[Theorem 4]{Takahashi2022-ml}, and \cite[Theorem 2]{Attouch2009-wf}.
\begin{theorem}[Convergence rates]\label[theorem]{thm:convergence-rate}
    Suppose that \cref{asmp:function-condition,asmp:feasibility,asmp:strongly-convex,asmp:global-subsequential} hold. Let $\{x^k\}_{k\in\NN}$ be a sequence generated by ABPG-VMA and let $\bar{x}=\lim_{k\to\infty}x^k$ be the stationary point given by \cref{thm:point-convergence}. Suppose further that $\Psi$ is a KL function with $\psi$ in the KL inequality~\cref{ineq:KL-property} at $\bar{x}$ taking the form $\psi(s) = cs^{1-\theta}$ for some $\theta\in[0,1)$ and $c > 0$. Then, the following statements hold:
    \begin{enumerate}
        \item If $\theta = 0$, then the sequence $\{x^k\}_{k\in\NN}$ converges to $\bar{x}$ in a finite number of iterations;
        \item If $\theta \in (0, 1/2]$, then there exist $c_1 > 0$ and $\gamma \in[0,1)$ such that $\|x^k - \bar{x}\| < c_1\gamma^k$ for all sufficiently large $k$;
        \item If $\theta \in (1/2, 1)$, then there exists $c_2 > 0$ such that $\|x^k - \bar{x}\| < c_2k^{-\frac{1-\theta}{2\theta - 1}}$ for all sufficiently large $k$.
    \end{enumerate}
\end{theorem}

\section{Numerical Experiments}\label{sec:numerical-experiments}
In this section, we conducted numerical experiments to examine the performance of our algorithm. All numerical experiments were performed in Python 3.9 on a MacBook Pro with an Apple M1 Max and 64GB LPDDR5 memory.

\subsection{\texorpdfstring{$\ell_p$}{lp}-Regularized Least Squares Problem}\label{subsec:lp-regularized-least-squares}
We consider the sparse $\ell_p$-regularized least squares problem, where $p$ is slightly larger than $1$,~\cite{doi:10.1137/18M1194456,7472557}:
\begin{align}
    \min_{x\in\RR^n}\quad \frac{1}{2}\|Ax-b\|^2 + \frac{\theta_p}{p}\|x\|^p_p,\label{lp-regular}
\end{align}
where $A\in\RR^{m\times n}$, $b\in\RR^m$, and $\theta_p>0$.
Let $g\equiv 0$. We also use $f$ and $\phi$ given by
\begin{align}
    f(x) = \frac{1}{2}\|Ax-b\|^2 + \frac{\theta_p}{p}\|x\|^p_p,\quad \text{and} \quad
    \phi(x) = \frac{1}{2}\|x\|^2 + \frac{\theta_p}{p}\|x\|^p_p.
\end{align}
Note that $f$ and $\phi$ are $\mathcal{C}^1$ if $p > 1$ while $\nabla f$ and $\nabla\phi$ are not globally Lipschitz continuous in general.
Although we can choose any $\lambda > 0$, we use $\lambda$ given by $\lambda < 1/L$ if $(f,\phi)$ is $L$-smad (see, for more details, \cref{remark:l-smad}). Note that our algorithm does not require the $L$-smad property.
\begin{proposition}[The $L$-smad property of $(f,\phi)$~{\cite[Proposition 24]{Takahashi}}]
    Let $f$ and $\phi$ be as defined above. Then, for any $L > 0$ satisfying
    \begin{align}
        L \geq \max\{\lambda_{\max}(A^\top A) + \theta_p, 1\}, \label{L-smad-L}
    \end{align}
    the functions $L\phi - f$ and $L\phi + f$ are convex on $\RR^n$, \ie, the pair $(f,\phi)$ is $L$-smad on $\RR^n$.
\end{proposition}
The subproblem of BPG cannot be solved in closed form if $p>1$ because its optimality condition is a nonlinear scalar equation that generally has no closed-form solution.
On the other hand, $\nabla^2 \phi(x) = I + \theta_p(p-1)\diag(|x|^{p-2})$ is a diagonal matrix and $\tilde{\mathcal{T}}_{\lambda}(x)$ can be solved in closed form even if $g\not\equiv 0$~\cite[Remark 25]{Takahashi}.

\begin{figure}[!tbp]
    \begin{minipage}[t]{0.49\linewidth}
        \centering
        \includegraphics[width=\linewidth]{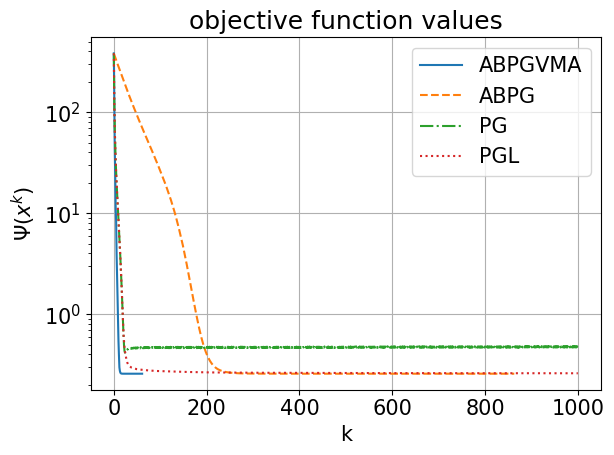}
        \subcaption{Objective function values ($p=1.2$)}
        \label{obj-iter-1.2}
    \end{minipage}
    \begin{minipage}[t]{0.49\linewidth}
        \centering
        \includegraphics[width=\linewidth]{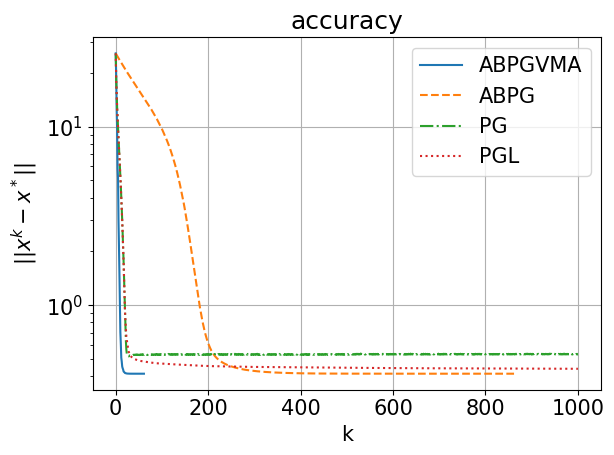}
        \subcaption{Accuracy ($p=1.2$)}
        \label{acc-iter-1.2}
    \end{minipage}
    \caption{Comparison with ABPG-VMA (blue), ABPG (orange), PG (green), and PGL (red) on the $\ell_p$ regularized least squares problem \cref{lp-regular}}
    \label{lp-result1}
\end{figure}

We compare ABPG-VMA with ABPG~\cite{Takahashi}, the proximal gradient algorithm (PG) with a constant step-size, and PG with line search (PGL).
We set $c_1 = 0.99$, $\eta = 0.9$, $\delta = 2$, and $t_{\max}=2^{20}$ for ABPG-VMA and $c_1 = 0.99$ and $\eta = 0.9$ for ABPG. Although $\nabla f$ is not Lipschitz continuous, PG uses the step-size $1/L$ given by~\cref{L-smad-L}. Note that PG does not guarantee global convergence. PGL searches for $\lambda_k > 0$ satisfying the descent lemma and uses the initial step-size $\lambda_0 = 1/L$ given by~\cref{L-smad-L}~\cite[p.283]{doi:10.1137/1.9781611974997}. The initial point $x^0\in\RR^n$ is generated from an \iid normal distribution. The maximum number of iterations is 1000. The terminal condition is $\|x^k-x^{k-1}\|\leq 10 ^{-8}$.

\begin{figure}[!tbp]
      \begin{minipage}[t]{0.49\linewidth}
        \centering
        \includegraphics[width=\linewidth]{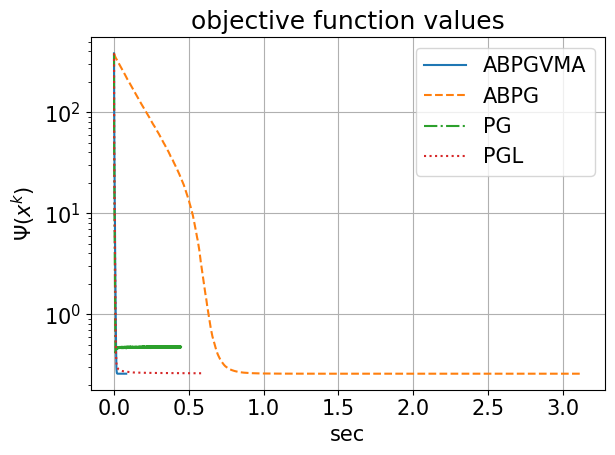}
        \subcaption{Objective function values ($p=1.2$)}
        \label{obj-time-1.2}
    \end{minipage}
    \begin{minipage}[t]{0.49\linewidth}
        \centering
        \includegraphics[width=\linewidth]{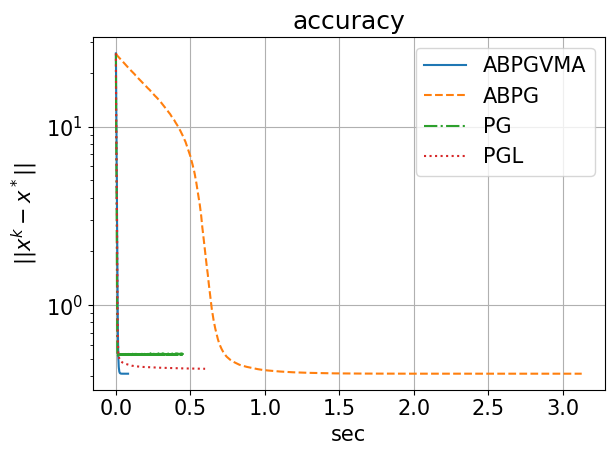}
        \subcaption{Accuracy ($p=1.2$)}
        \label{acc-time-1.2}
    \end{minipage}
    \caption{Comparison with ABPG-VMA (blue), ABPG (orange), PG (green), and PGL (red) on the $\ell_p$ regularized least squares problem \cref{lp-regular}}
    \label{lp-result2}
\end{figure}

The problem setting is as follows. We generate the matrix $A\in\RR^{m\times n}$ from an \iid normal distribution and normalize each column of $A$ to have unit $\ell_2$-norm. The ground truth $x^*\in\RR^n$ has 10\% nonzero entries at randomly chosen positions, whose values are drawn from an \iid normal distribution, and is normalized so that $\|x^*\|=1$. We set $b=Ax^*$. For $(m,n)=(1000,700)$, $p=1.2$, and $\theta_p=0.1$, \cref{lp-result1} shows the objective function value $\Psi(x^k)$ and the accuracy $\|x^k-x^*\|$ at each iteration on a logarithmic scale and~\cref{lp-result2} shows those on the time axis. When $p = 1.2$, the gradient of $\|x\|^p_p$ is not Lipschitz continuous on $(-1,1)^n$. This is why PG and PGL are not guaranteed to converge to a stationary point in this setting.
Under this condition, only ABPG-VMA and ABPG converge within 1000 iterations, while PG and PGL do not satisfy the stopping condition. In particular, ABPG-VMA meets the stopping condition in 61 iterations, which is substantially fewer than ABPG, which requires 863 iterations.
\cref{table:abpg-vma-statistics} shows the mean and median of step-sizes and update-selection statistics for ABPG-VMA.
\begin{table}[!htbp]
    \centering
    \caption{Step-size and update-selection statistics for ABPG-VMA.
    The four Frequency (\%) columns report the percentages of iterations classified by the update selected in
    \cref{line:ABPG-VMA-min-obj} of \cref{alg:ABPG-VMA}:
    the first three correspond to selecting $x^k+t_kd^k$ with $t_k<1$, $t_k=1$, and $t_k>1$, respectively,
    and the last to selecting $y^k$ when $\Psi(y^k)<\Psi(x^k+t_kd^k)$.
    These categories are mutually exclusive, and their percentages sum to $100\%$ up to rounding.}
    \label{table:abpg-vma-statistics}
    \begin{tabular}{lrrrrrrrr}
        \toprule
        & & & \multicolumn{2}{c}{Step-size $t_k$} & \multicolumn{4}{c}{Frequency (\%)} \\
        \trimmedcmidrule{4-5}
        \trimmedcmidrule{6-9}
        Problem & $m$ & $n$ & Mean & Median & $t_k<1$ & $t_k=1$ & $t_k>1$ & $y^k$ \\
        \midrule
        & 1000  & 700 & 1.43  & 1.00 & 0.0  & 63.9 & 36.1 & 0.0 \\
        & 100  & 1500 & 10.86 & 4.00 & 0.1  & 2.1  & 97.8 & 0.0 \\
        & 100  & 3000 & 16.84 & 8.00 & 0.1  & 2.9  & 97.0 & 0.0 \\
        \cref{lp-regular}
        & 100  & 5000 & 13.12 & 8.00 & 0.0  & 1.7  & 98.3 & 0.0 \\
        & 1000 & 1500 & 2.22  & 2.00 & 0.0  & 30.3 & 69.5 & 0.2 \\
        & 1000 & 3000 & 3.04  & 2.00 & 0.0  & 2.4  & 97.6 & 0.0 \\
        & 1000 & 5000 & 4.74  & 4.00 & 0.0  & 3.4  & 96.6 & 0.0 \\
        \midrule
        \cref{KL-opt}
        & 500  & 200  & 0.61  & 0.90 & 44.4 & 7.8  & 0.2  & 47.6 \\
        \midrule
        \cref{PR_problem}
        & 1000 & 200  & $6.49\times10^4$ & $3.28\times10^4$ & 0.0 & 0.0 & 100.0 & 0.0 \\
        \bottomrule
    \end{tabular}
\end{table}

Next, we show the average performance of the four methods---ABPG-VMA, ABPG, PG, and PGL---on the $\ell_p$-regularized least squares problem. Specifically, we selected combinations of $m$ and $n$ from the set $\{100,1000\}\times \{1500,3000,5000\}$. ... For each combination, we generated 100 instances for $m=100$ and 10 instances for $m=1000$: in each instance, we generated $A$ and $x^*$ with $10\%$ nonzero entries in the same manner as above. For each generated instance, we set $b=Ax^*$, $p=1.2$, and $\theta_p=0.1$, and capped the number of iterations at 1000.
\cref{table_avg} presents the average performance, including the number of iterations, the accuracy of the recovered point, the objective values, and computation time, across different instances. ABPG-VMA outperformed ABPG, PG, and PGL. Moreover, ABPG-VMA converged in fewer iterations and in a shorter amount of time than ABPG, PG, and PGL.

For ABPG-VMA, we varied $\lambda$ by scaling $L$ by each value in $\{0.01, 0.1, 1, 10, 100\}$, where $L$ is given by~\cref{L-smad-L}, and compared the results. The results are shown in \cref{fig:l-smad-varying}. When the value around $L$ is chosen, both the number of iterations and the computation time are reduced.

\begin{figure}[!tbp]
    \centering
    \includegraphics[width=0.8\linewidth]{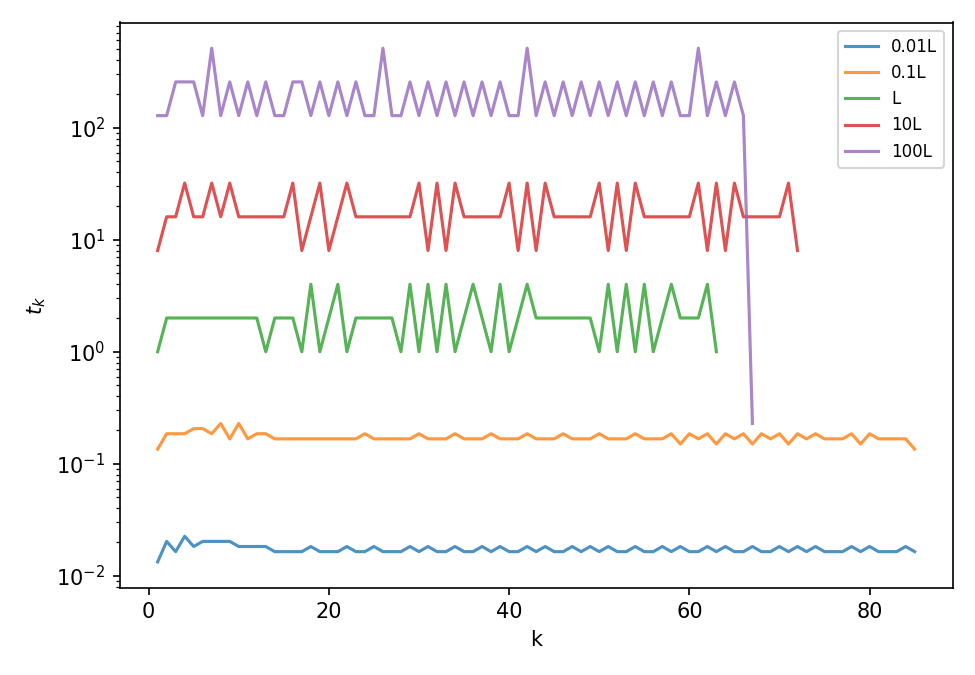}
    \caption{Comparative analysis of algorithmic performance with varying $L$-smad parameter scaling}
    \label{fig:l-smad-varying}
\end{figure}

\begin{table}[htbp]
  \centering
  \captionof{table}{Average number of iterations, objective function value, accuracy, and CPU time for ABPG-VMA, ABPG, PG, and PGL using random instances of the $\ell_p$-regularized least squares problem \cref{lp-regular} ($100$ instances for $m=100$, and $10$ instances for $m=1000$). The ABPG, PG, and PGL rows are unchanged from the original protocol (their implementations were not modified in this revision); only the ABPG-VMA row was recomputed with the updated line search.
  Objective values, accuracy values, and CPU times are reported to three decimal places.
  Bold indicates the smallest displayed value of each metric for each $(m,n)$, with ties included.}
  \label{table_avg}
  \centering
  \begin{tabular}{lllrrrr}
    \toprule
    $m$ & $n$ & algorithm & iteration & obj & acc & time \\
    \midrule
    \multirow[t]{12}{*}{100} & \multirow[t]{4}{*}{1500} & ABPG-VMA & \textbf{122} & \textbf{0.160} & \textbf{0.956} & \textbf{0.156} \\
     &  & ABPG & 847 & 0.494 & 0.993 & 1.022 \\
     &  & PG & 1000 & 6.218 & 1.068 & 0.624 \\
     &  & PGL & 1000 & 17.968 & 4.385 & 0.686 \\
    \cmidrule{2-7}
     & \multirow[t]{4}{*}{3000} & ABPG-VMA & \textbf{115} & \textbf{0.143} & \textbf{0.989} & \textbf{0.533} \\
     &  & ABPG & 980 & 0.503 & 0.997 & 3.750 \\
     &  & PG & 1000 & 6.412 & 1.039 & 3.174 \\
     &  & PGL & 1000 & 43.325 & 7.374 & 3.285 \\
    \cmidrule{2-7}
     & \multirow[t]{4}{*}{5000} & ABPG-VMA & \textbf{237} & \textbf{0.130} & \textbf{0.999} & \textbf{1.951} \\
     &  & ABPG & 1000 & 0.492 & \textbf{0.999} & 6.497 \\
     &  & PG & 1000 & 6.372 & 1.025 & 5.618 \\
     &  & PGL & 1000 & 79.301 & 10.345 & 5.799 \\
    \midrule
    \multirow[t]{12}{*}{1000} & \multirow[t]{4}{*}{1500} & ABPG-VMA & \textbf{55} & \textbf{0.309} & \textbf{0.581} & \textbf{0.277} \\
     &  & ABPG & 587 & 0.496 & 1.000 & 3.473 \\
     &  & PG & 1000 & 28.745 & 2.067 & 0.777 \\
     &  & PGL & 1000 & 9.874 & 2.326 & 1.037 \\
    \cmidrule{2-7}
     & \multirow[t]{4}{*}{3000} & ABPG-VMA & \textbf{68} & \textbf{0.338} & \textbf{0.733} & \textbf{0.575} \\
     &  & ABPG & 504 & 0.512 & 1.000 & 7.251 \\
     &  & PG & 1000 & 36.533 & 1.958 & 3.775 \\
     &  & PGL & 1000 & 24.393 & 3.981 & 4.343 \\
    \cmidrule{2-7}
     & \multirow[t]{4}{*}{5000} & ABPG-VMA & \textbf{59} & \textbf{0.344} & \textbf{0.832} & \textbf{0.923} \\
     &  & ABPG & 434 & 0.498 & 1.000 & 9.925 \\
     &  & PG & 1000 & 43.081 & 1.875 & 6.493 \\
     &  & PGL & 1000 & 45.076 & 5.704 & 7.414 \\
    \bottomrule
\end{tabular}
\end{table}

\subsection{Nonnegative Linear Inverse Problem}\label{subsec:nonnegative-linear-inverse-problem}
Given a nonnegative matrix $A\in\RR^{m\times n}_+$ and a nonnegative vector $b\in\RR^m_+$, the goal of nonnegative linear inverse problems is to recover a signal $x \in \RR^n_+$ such that $Ax \simeq b$.
Nonnegative linear inverse problems have been studied in image deblurring~\cite{Bertero2009-jq} and positron emission tomography~\cite{Vardi1985-de}, as well as in optimization~\cite{Bauschke2017-hg,Takahashi}.
To achieve the goal of nonnegative linear inverse problems, we focus on the convex optimization problem given by
\begin{align}
    \min_{x\in\RR^n_+} \quad D_{\rm{KL}}(Ax, b) + \theta_1\|x\|_1,\label{KL-opt}
\end{align}
where the Kullback--Leibler divergence is defined as follows:
\begin{align}
    D_{\rm{KL}}(x,y) = \sum_{i=1}^m \left( x_i \log \frac{x_i}{y_i} + y_i - x_i\right).
\end{align}
Let $f(x) = D_{\rm{KL}}(Ax,b)$ and $g(x)=\theta_1\|x\|_1$.
We use $\phi_{0}(x)=\sum_{i=1}^{n}x_i\log x_i$ as the kernel generating distance for BPG and $\phi_{1}(x)=\phi_{0}(x)+\frac{1}{2}\|x\|^2$ as the kernel generating distance for our algorithm and ABPG. In this case, we also define $C=\interior\dom\phi_{0}=\interior\dom\phi_{1}=\RR^n_{++}$. When $\sum_{i=1}^ma_{ij}=1$, the pair $(f,\phi_{0})$ is $1$-smad \cite{Bauschke2017-hg} and the pair $(f,\phi_{1})$ is also $1$-smad~\cite{Takahashi}.
We compare ABPG-VMA with ABPG~\cite{Takahashi}, PGL, and BPG. Those subproblems can be solved in closed form.

The problem setting is as follows. We generate the entries of a matrix $\tilde{A}\in\RR^{m\times n}$ by taking the absolute values of \iid samples from a normal distribution, and obtain $A\in\RR^{m\times n}_+$ by normalizing each column of $\tilde{A}$ so that $\sum_{i=1}^m a_{ij}=1$; hence the pairs $(f,\phi_0)$ and $(f,\phi_1)$ are 1-smad.
The ground truth $x^*\in\RR^n_+$ has 5\% nonzero entries at randomly chosen positions, whose values are drawn from the uniform distribution on $[0,1)$, and is normalized so that $\sum_{j=1}^n x^*_j=1$.
We set $b=Ax^*$.
The initial point $x^0\in\RR^n_{++}$ is obtained by taking the absolute values of \iid samples from a normal distribution and normalizing them so that $\sum_{j=1}^n x^0_j=1$.
For $(n,m)=(200,500)$ and $\theta_1=0.05$, \cref{nonneg-result1} shows the objective function value $\Psi(x^k)$ and the accuracy $\|x^k-x^*\|$ at each iteration on a logarithmic scale and~\cref{nonneg-result2} shows those on the time axis.

We set $c_1 = 0.99$, $\eta = 0.9$, $\delta = 2$, and $t_{\max}=2^{20}$ for ABPG-VMA and $c_1 = 0.99$ and $\eta = 0.9$ for ABPG.
Under this condition, ABPG-VMA outperforms the other three methods in terms of the reduction in the objective function value per iteration and per unit time. It is also observed that the objective function values obtained after 1000 iterations are comparable across all methods. Notably, the error with respect to the true value is significantly smaller for ABPG-VMA than for the other three methods.

\begin{figure}[!tbp]
  \begin{minipage}[b]{0.49\linewidth}
    \centering
    \includegraphics[width=\linewidth]{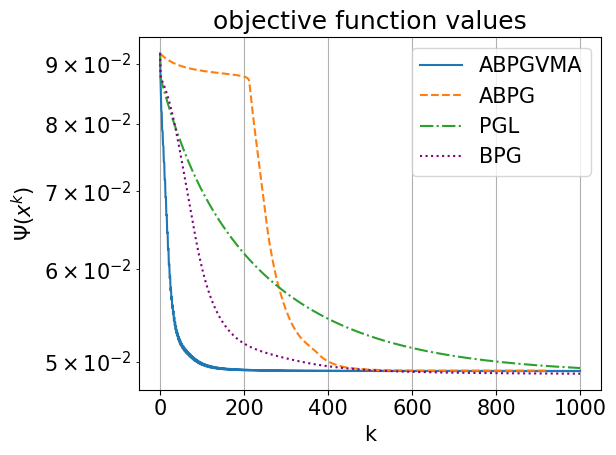}
    \subcaption{Objective function values}
  \end{minipage}
  \begin{minipage}[b]{0.49\linewidth}
    \centering
    \includegraphics[width=\linewidth]{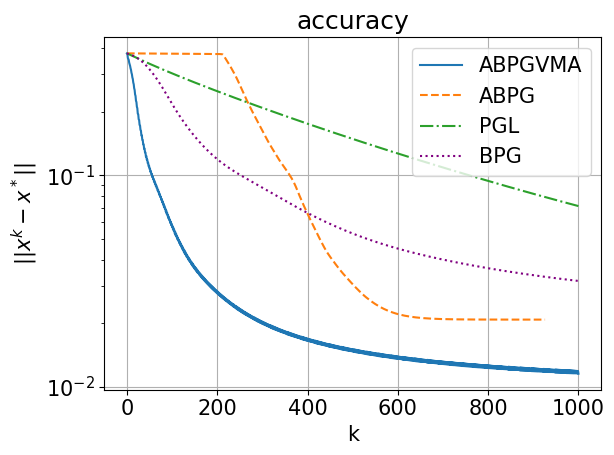}
    \subcaption{Accuracy}
  \end{minipage}
  \caption{Comparison with ABPG-VMA (blue), ABPG (orange), PGL (green), and BPG (purple) on the nonnegative linear inverse problem \cref{KL-opt}}
  \label{nonneg-result1}
\end{figure}
\begin{figure}[!tbp]
  \begin{minipage}[b]{0.49\linewidth}
    \centering
    \includegraphics[width=\linewidth]{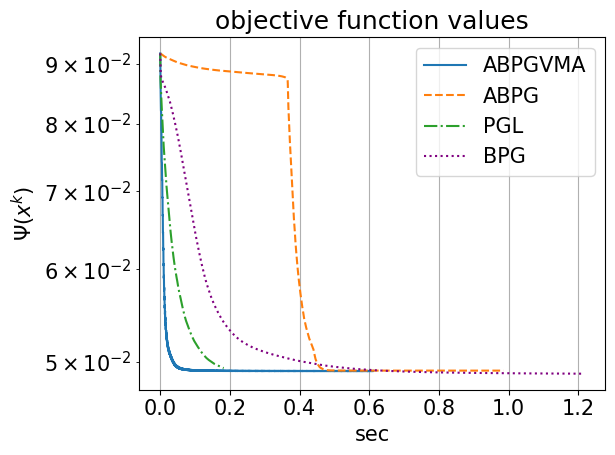}
    \subcaption{Objective function values}
  \end{minipage}
  \begin{minipage}[b]{0.49\linewidth}
    \centering
    \includegraphics[width=\linewidth]{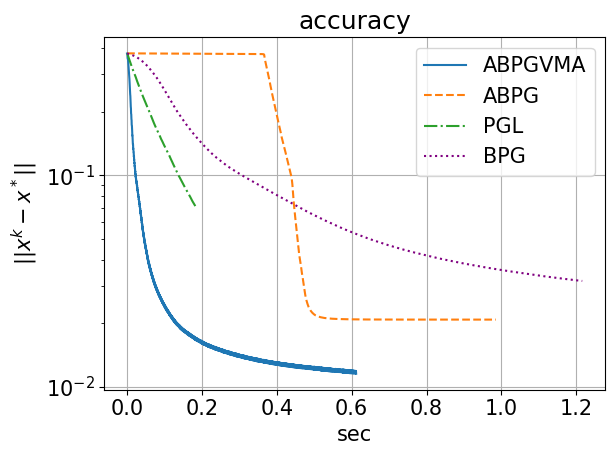}
    \subcaption{Accuracy}
  \end{minipage}
  \caption{Comparison with ABPG-VMA (blue), ABPG (orange), PGL (green), and BPG (purple) on the nonnegative linear inverse problem \cref{KL-opt}}
  \label{nonneg-result2}
\end{figure}

\subsection{Phase Retrieval}
We consider phase retrieval, \ie, recovering a signal $x \in \RR^n$ such that $|\langle a_i, x\rangle|^2 \simeq b_i$ for $i = 1, \ldots, m$, where $a_i\in\RR^n$ describes the model and $b_i\in\RR$ is an observed magnitude. Phase retrieval has been studied in many applications, such as image processing~\cite{Candes2015-mk} and X-ray crystallography~\cite{Patterson1934-sg,Patterson1944-nm} as well as optimization~\cite{Bolte-2018,Takahashi2022-ml,Takahashi2026-ch}. We address the following nonconvex optimization problem:
\begin{align}
    \min_{x \in \RR^n} \quad \frac{1}{4}\sum_{i=1}^m\left(|\langle a_i, x\rangle|^2 - b_i\right)^2.\label{PR_problem}
\end{align}
Let $f(x) = \frac{1}{4}\sum_{i=1}^m\left(|\langle a_i, x\rangle|^2 - b_i\right)^2$. When we use $\phi(x) = \frac{1}{4}\|x\|^4 + \frac{1}{2}\|x\|^2$, $(f,\phi)$ is $L$-smad for any $L \geq \sum_{i=1}^m\left(3\|a_i\|^4 + \|a_i\|^2|b_i|\right)$ (see \cite[Lemma 5.1]{Bolte-2018}).
We compare ABPG-VMA with ABPG, PGL, and BPG. Those subproblems can be solved in closed form (see~\cite[Proposition 5.1]{Bolte-2018} for the subproblem of BPG). We set $\lambda = 1/L$, where $L = \sum_{i=1}^m\left(3\|a_i\|^4 + \|a_i\|^2|b_i|\right)$.

The problem setting is as follows. We generate $a_i\in\RR^n$, $i = 1,\ldots,m$, from an \iid normal distribution.
The ground truth $x^*\in\RR^n$ is generated as follows: we draw a vector from an \iid normal distribution, set 5\% of its entries to zero, take the absolute value of each entry, and normalize it so that $\|x^*\|=1$.
We set $b_i=|\langle a_i, x^*\rangle|^2$.
The initial point $x^0$ is given by the spectral initialization, \ie, the normalized leading eigenvector of $\sum_{i=1}^m b_i a_i a_i^\top$ computed by the power method.
Since $x^*$ and $-x^*$ yield the same measurements $b_i$, the global sign of $x^*$ cannot be recovered.
We thus measure the accuracy by $\dist(x^k,\{x^*,-x^*\}) = \min\{\|x^k-x^*\|, \|x^k+x^*\|\}$.
For $(n,m)=(200,1000)$, \cref{fig:phase-retrieval-iter} shows the objective function value $f(x^k)$ and the accuracy $\dist(x^k,\{x^*,-x^*\})$ at each iteration on a logarithmic scale, and \cref{fig:phase-retrieval-time} shows those on the time axis.

We set $c_1 = 0.99$, $\eta = 0.9$, $\delta = 2$, and $t_{\max}=2^{20}$ for ABPG-VMA and $c_1 = 0.99$ and $\eta = 0.9$ for ABPG.
Under this condition, ABPG-VMA is the only method that converges to a near-exact solution within 1000 iterations. ABPG-VMA outperforms the other three methods in terms of the reduction in the objective function value per iteration and per unit time. ABPG-VMA allows large step-sizes while ABPG and BPG use $\lambda = 1/L$, which can be small, and PGL would estimate small step-sizes without Lipschitz continuous gradients. Moreover, the distance $\dist(x^k,\{x^*,-x^*\})$ to the ground truth is significantly smaller for ABPG-VMA than for the other three methods.

\begin{figure}[!tbp]
    \begin{minipage}[t]{0.49\linewidth}
        \centering
        \includegraphics[width=\linewidth]{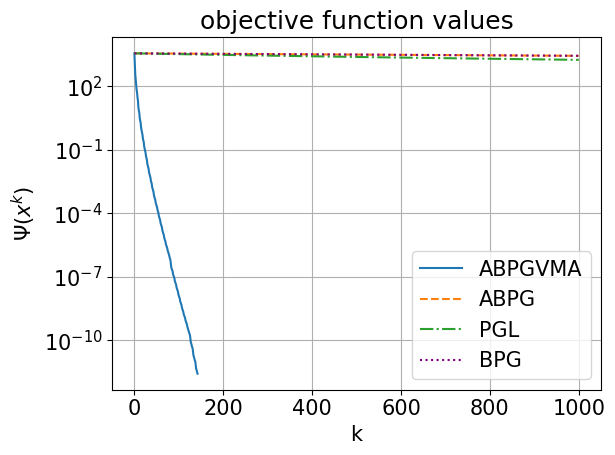}
        \subcaption{Objective function values}
    \end{minipage}
    \begin{minipage}[t]{0.49\linewidth}
        \centering
        \includegraphics[width=\linewidth]{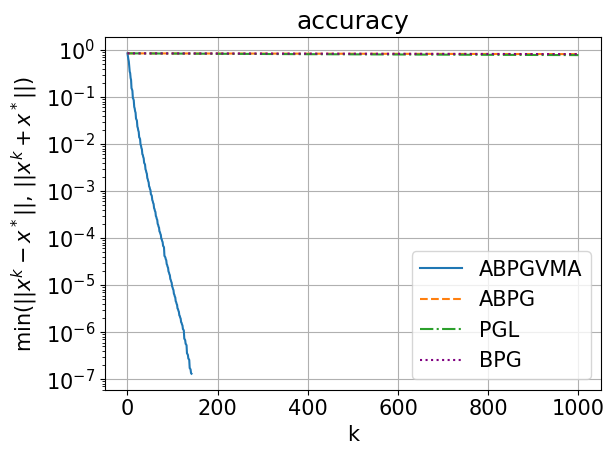}
        \subcaption{Accuracy}
    \end{minipage}
    \caption{Comparison with ABPG-VMA (blue), ABPG (orange), PGL (green), and BPG (purple) on the phase retrieval \cref{PR_problem}}
    \label{fig:phase-retrieval-iter}
\end{figure}

\begin{figure}[!tbp]
    \begin{minipage}[t]{0.49\linewidth}
        \centering
        \includegraphics[width=\linewidth]{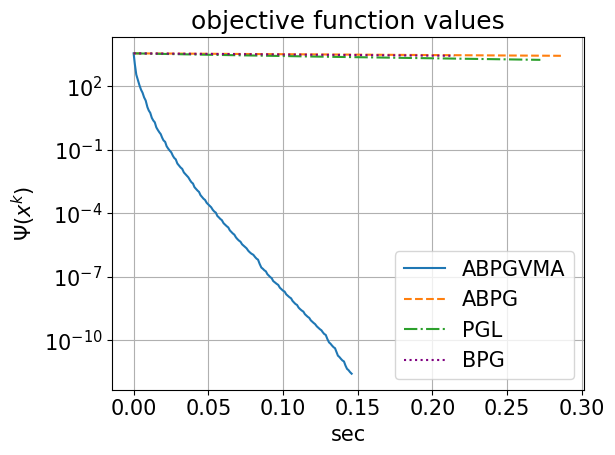}
        \subcaption{Objective function values}
    \end{minipage}
    \begin{minipage}[t]{0.49\linewidth}
        \centering
        \includegraphics[width=\linewidth]{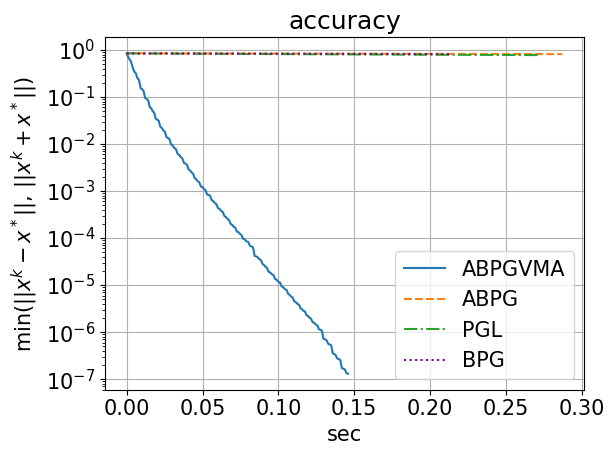}
        \subcaption{Accuracy}
    \end{minipage}
    \caption{Comparison with ABPG-VMA (blue), ABPG (orange), PGL (green), and BPG (purple) on the phase retrieval \cref{PR_problem}}
    \label{fig:phase-retrieval-time}
\end{figure}

\section{Conclusion}\label{sec:conclusion}
In this paper, we propose the approximate Bregman proximal gradient algorithm with variable metric Armijo line search (ABPG-VMA) for composite nonconvex optimization problems. Our line search condition allows a larger step-size than existing algorithms.
We have established global subsequential convergence under the level-boundedness and local smoothness assumptions. Moreover, under the KL property, we have proved global convergence to a stationary point even when $g \not\equiv 0$. To the best of our knowledge, this is the first global convergence result for ABPG-type algorithms in this setting.
Moreover, our numerical experiments on $\ell_p$ regularized least squares problems, nonnegative linear inverse problems, and phase retrieval have shown that ABPG-VMA outperforms ABPG and proximal gradient algorithms. 

\section*{Declarations}
\textbf{Funding}: This project has been funded by the Japan Society for the Promotion of Science (JSPS) and the Nakajima foundation; JSPS KAKENHI Grant Numbers JP23K19953 and JP25K21156; JSPS KAKENHI Grant Number JP23H03351.\\
\textbf{Conflict of interest}:
The authors have no competing interests to declare that are relevant to the content of this article.\\
\textbf{Data availability}:
The datasets generated during and/or analyzed during the current study are available in the GitHub repository, \url{https://github.com/ShotaTakahashi/ApproximateBPG}.

\addcontentsline{toc}{section}{References}
\bibliography{main}

\end{document}